\documentclass[preprint,12pt]{elsarticle}

\usepackage{amssymb}
\usepackage{amsmath}
\usepackage{amsthm}

\usepackage{paralist}
\usepackage{color}
\usepackage{hyperref}

\theoremstyle{definition}
\newtheorem{theorem}{Theorem}[section]
\newtheorem{proposition}[theorem]{Proposition}

\newtheorem{example}[theorem]{Example}
\newtheorem{definition}[theorem]{Definition}
\newtheorem{lemma}[theorem]{Lemma}
\newtheorem{remark}[theorem]{Remark}
\newtheorem{definition-proposition}[theorem]{Definition-Proposition}

\newcommand{\rk}{\mathrm{rk}\,}
\newcommand{\llangle}{\langle\!\langle}
\newcommand{\rrangle}{\rangle\!\rangle}

\newcommand{\set}[1]{\left\lbrace #1\right\rbrace}

\newcommand{\simto}{\xrightarrow{\raisebox{-0.7ex}[0ex][0ex]{$\sim$}}} 

\begin{document}

\numberwithin{equation}{section}

\begin{frontmatter}



\title{Port-Hamiltonian formulation of energy-based modeling frameworks}


\author{Jonas Kirchhoff} 

\affiliation{organization={Institut für Mathematik, Martin-Luther-Universität Halle-Wittenberg},
            addressline={Theodor-Lieser-Stra\ss e 5}, 
            city={Halle (Saale)},
            postcode={06120}, 
            state={Sachsen-Anhalt},
            country={Germany}}

\begin{abstract}

The energy-based modelling framework proposed by Altmann and Schulze is demonstrated to have two distinct representations as port-Hamiltonian systems. The ambiguity is expressed in the r\^{o}le of the algebraic variable, and is distinct from the usual ambiguity in the geometric representation of coordinate representations of classical port-Hamiltonian systems. A straightforward extension to systems with feedthrough is given. 

\end{abstract}


\end{frontmatter}




\section{Introduction}

Classes of systems based on some form of ``energy'', which are guided by balance equations and an underlying network structure, are very well-known. One of the most widely used of such system classes is the class of port-Hamiltonian systems. This system class is geometric in nature with the underlying geometric structure given by the so-called \textit{Dirac structures}, which may be viewed as models of lossless transport of energy between partial systems. These partial systems are grouped by their effect on the energy into conservative parts, dissipative parts and connective parts, where the system is open to interconnection. The conservative parts of the system are modelled by a Lagrangian submanifold in the cotangent bundle of the underlying manifold of ``states''\footnote{Here, manifold of ``states'' is to be taken with caution, since this could very well be the phase space of a mechanical system, where the second order differential equation becomes a first order system.}, and the dissipative part is modelled by a linear non-negative Lagrangian relation. The connective parts are external ports expressed in free variables in the Dirac structure. It should be stressed that the three system parts, or types of system parts, do not correspond to real-existing partial systems, e.g. is a real-existing spring virtually split into an ideal spring and an ideal damper.

Given several systems, the energy leaving (or entering) the systems via their external ports may be routed lossless through a Dirac structure, and the resulting interconnected system has again the structure of a port-Hamiltonian system. This structure-preserving interconnection, which allows for automated modeling, and the energy-based modeling are two key features of port-Hamiltonian systems. 

Recently, Altmann and Schulze proposed a class of energy-based models as extensions of port-Hamiltonian systems in~\cite{AltmSchu25} with further results on discretization and model reduction in~\cite{AltmKarsSchu26}. In this note, we demonstrate that this modelling class admits two different (geometric) representations as port-Hamiltonian systems. We extend the system class to include feedthrough.

\section{Reminder on port-Hamiltonian systems}

In this section, we recall the elementary definitions of Dirac structure, resistive Lagrangian subspace and Lagrangian submanifold, and we recall the interconnection. For simplicity of understanding, we restrict ourself to open sets and constant vector spaces. The extension to general (finite-dimensional, smooth) manifolds and (finite-dimensional, smooth) vector bundles is straightforward.

\begin{definition}[{cf.~\cite[Definition 1.1.1]{Cour90}}]
A \textit{Dirac structure} in a real, finite-dimensional vector space $X$ with dual space $X^*$ is a subvector space $D\leq X\oplus X^*$, which is isotropic and co-isotropic with respect to the indefinite inner product
\begin{align*}
\llangle\cdot,\cdot\rrangle: (X\oplus X^*)\otimes (X\oplus X^*) & \to \mathbb{R},\\
\llangle(f,e),(f',e')\rrangle & := \langle e,f'\rangle+\langle e',f\rangle,
\end{align*}
where $\langle\cdot,\cdot\rangle$ denotes the duality product.
\end{definition}

The standard example of Dirac structures are co-graphs $\set{(J\alpha,\alpha)}$ and graphs $\set{(X,\Omega X)}$ for antisymmetric linear maps $\Omega:X\to X^*$ and $J:X^*\to X$. It is well-known that, due to the non-degeneracy of $\llangle\cdot,\cdot\rrangle$, the co-isotropy of isotropic subspaces $D$ is characterised by $\dim D = \dim V$.

\begin{definition}[{cf.~\cite[Definitions I.6.4]{LibeMarl87}}]
A \textit{Lagrangian subspace} in a real, finite-dimensional vector space $X$ with dual space $X^*$ is a subvector space $L\leq X\oplus X^*$, which is isotropic and co-isotropic with respect to the canonical symplectic form
\begin{align*}
\Omega_{\mathrm{can}}: (X\oplus X^*)\otimes (X\oplus X^*) & \to \mathbb{R},\\
\Omega_{\mathrm{can}}\big((f,e),(f',e')\big) & := \langle e,f'\rangle-\langle e',f\rangle.
\end{align*}
The subspace $L$ is \textit{non-negative}, if $\langle e,f\rangle\geq 0$ for all $(f,e)\in L$.
\end{definition}

The standard example of Lagrangian subspaces are graphs $\set{(P\alpha,\alpha)}$ and $\set{(X,RX)}$ for symmetric linear maps $R:X\to X^*$ and $P:X^*\to X$, which are non-negative if, and only if, $P$ and $R$ are positive semi-definite. 

The last geometric ingredient of port-Hamiltonian systems that we recall here are the Lagrangian submanifold in the cotangent bundle of an open set, which is trivialised. 

\begin{definition}
A \textit{Lagrangian submanifold} in an open set $U\subseteq X$, $X$ a finite-dimensional real vector space, is a \textit{submanifold\footnote{cf.~\cite[Definition 1.5.3]{Conlon93}}} $\mathcal{L}\subseteq U\times X^*$ so that each fibre of its tangent bundle is a Lagrangian submanifold.
\end{definition}

Equivalently, a Lagrangian submanifold is characterised as a submanifold $\mathcal{L}$ with $\dim\mathcal{L} = \dim X$ so that for any two smooth curves $(x,\alpha),(x',\alpha'):(-\varepsilon,\varepsilon)\to \mathcal{L}$ with $(x,\alpha)(0) = (x',\alpha')(0)$, 
\begin{align*}
\langle\tfrac{\mathrm{d}}{\mathrm{d}t}\alpha'\vert_{t = 0},\tfrac{\mathrm{d}}{\mathrm{d}t}x\vert_{t = 0}\rangle = \langle\tfrac{\mathrm{d}}{\mathrm{d}t}\alpha\vert_{t = 0},\tfrac{\mathrm{d}}{\mathrm{d}t}x'\vert_{t = 0}\rangle.
\end{align*}
With these, ingredients, we may give the geometric definition of the port-Hamiltonian systems.

\begin{definition}[{cf.~\cite{MascScha20}}]
A \textit{port-Hamiltonian system} is the dynamical system
\begin{equation}\label{eq:pH_orig}
\begin{aligned}
\big(\tfrac{\mathrm{d}}{\mathrm{d}t}x,f_r,-y,e,e_r,u\big) & \in\mathcal{D}\\
(f_r,e_r) & \in L_r\\
(x,e) & \in\mathcal{L}
\end{aligned}
\end{equation}
with
\begin{itemize}
\item finite-dimensional, real vector spaces $X,X_r,X_p$
\item smooth solution curves $x\in\mathcal{C}^\infty\big((t_0,t_1),U\big)$, $U\subseteq X$ open,
\item Dirac structure $D\leq (X\oplus X_r\oplus X_p)\oplus(X^*\oplus X_r^*\oplus X_p^*)$,
\item non-negative Lagrangian subspace $L_r\leq X_r\oplus X_r^*$
\item Lagrangian submanifold $\mathcal{L}\subseteq U\times X^*$
\end{itemize}
\end{definition}

Choose now a fixed set of coordinates for each of the spaces $X,X_r,X_p$, and consider the case that $\mathcal{L} = \set{\big(x,\partial_{x}H(x)\big)}$ for a smooth function $H$, which is a Lagrangian submanifold by~\cite[Proposition III.2.4]{LibeMarl87}. Then, the image (basis) and kernel (co-basis) representations of Lagrangian and Dirac subspaces yields that port-Hamiltonian systems~\eqref{eq:pH_orig} may be written as
\begin{align*}
E\begin{pmatrix}
\tfrac{\mathrm{d}}{\mathrm{d}t}x\\
F_rz\\
-y
\end{pmatrix} = F\begin{pmatrix}
\partial_x H(x)\\
E_rz\\
u
\end{pmatrix}
\end{align*}
with quadratic matrices $E,F,F_r,E_r$ of appropriate dimension satisfying the structural conditions
\begin{align*}
EF^\top+FE^\top = 0,\quad E_r^\top F_r - F_r^\top E_r = 0, \quad E_r^\top F_r\geq 0,
\end{align*}
and $\rk[E,F] = \dim (X\oplus X_r\oplus X_p)$ and $\rk[F_r^\top,E_r^\top] = \dim X_r$. In that case, the port-Hamiltonian systems are said to have an \textit{explicitly defined energy}, otherwise, the energy of the system is \textit{implicitly defined}.

\section{System class under consideration}

In~\cite{AltmSchu25}, the authors propose a ``novel energy-based modeling framework'', which, just as port-Hamiltonian systems, admits energy balance equations, structure-preserving interconnections and discretizations. The system class under consideration are written in coordinates as
\begin{equation}\label{eq:desired_systems}
\begin{aligned}
\begin{pmatrix}
\partial_{z_1}H(z_1,z_2)\\
\tfrac{\mathrm{d}}{\mathrm{d}t}z_2\\0
\end{pmatrix} & = (J-R)\begin{pmatrix}
\tfrac{\mathrm{d}}{\mathrm{d}t}z_1\\\partial_{z_2}H(z_1,z_2)\\z_3
\end{pmatrix} + Bu\\
y & = B^\top \begin{pmatrix}
\tfrac{\mathrm{d}}{\mathrm{d}t}z_1\\\partial_{z_2}H(z_1,z_2)\\z_3
\end{pmatrix}
\end{aligned}
\end{equation}
with 
\begin{itemize}
\item solution curves $(z_1,z_2,z_3)\in\mathcal{C}^\infty\big((t_0,t_1),\mathbb{R}^{n_1+n_2+n_3})$
\item Hamiltonian function $H\in\mathcal{C}^\infty(\mathbb{R}^{n_1+n_2})$
\item matrices $J,R$ and $B$ of appropriate dimensions satisfying the structural conditions
\begin{align*}
J+J^\top = R-R^\top = 0,\qquad R\geq 0.
\end{align*}
\end{itemize}
For ease of readability, we shall omit the argument of $\partial_{z_i}H(z_1,z_2)$ and write $\partial_{z_i}H$ whenever the argument is unambiguous. Moreover, we shall consider the case that $H\in\mathcal{C}^\infty(U)$ for some open $U\subseteq\mathbb{R}^{n_1+n_2}$. In case that $n_3 = 0$, the energy-stable port-Hamiltonian systems without feedthrough, which are introduced by~\cite{BuchGlasZwar26}, emerge from~\eqref{eq:desired_systems}. Their port-Hamiltonian nature is seen as follows.

\begin{proposition}\label{prop:es_pH}
Let $n_1,n_2,m\in\mathbb{N}_0$ and let $U\subseteq\mathbb{R}^{n_1+n_2}$ be an open and non-empty set. Consider matrices $J,R\in\mathbb{R}^{(n_1+n_2)\times(n_1+n_2)}$ and $B\in\mathbb{R}^{(n_1+n_2)\times m}$ so that $J$ is antisymmetric and $R$ is symmetric and positive semi-definite, and consider $H\in\mathcal{C}^\infty(U)$. The dynamical system
\begin{equation}\label{eb_1}
\begin{aligned}
\begin{pmatrix}
\partial_{z_1}H\\
\tfrac{\mathrm{d}}{\mathrm{d}t}z_2
\end{pmatrix} & = (J-R)\begin{pmatrix}
\tfrac{\mathrm{d}}{\mathrm{d}t}z_1\\\partial_{z_2}H
\end{pmatrix} + Bu\\
y & = B^\top \begin{pmatrix}
\tfrac{\mathrm{d}}{\mathrm{d}t}z_1\\\partial_{z_2}H(z_1,z_2)
\end{pmatrix}
\end{aligned}
\end{equation}
is equivalent to a port-Hamiltonian system.
\end{proposition}
\begin{proof}
Define the $2(n_1+n_2)+m$-dimensional subspace $\mathcal{D}$ by the equation
\begin{align*}
\begin{pmatrix}
\begin{pmatrix}\alpha_1\\X_2\end{pmatrix}\\f\\y
\end{pmatrix} = \begin{bmatrix}
J & I & B\\
-I & 0 & 0\\
-B^\top & 0 & 0
\end{bmatrix}\begin{pmatrix}
\begin{pmatrix}X_1\\\alpha_2\end{pmatrix}\\e\\u
\end{pmatrix}.
\end{align*}
Then, for all $(X_1,X_2,f,y,\alpha_1,\alpha_2,e,u)\in\mathcal{D}$
\begin{align*}
\begin{pmatrix}
\alpha_1\\\alpha_2\\e\\u
\end{pmatrix}^\top\begin{pmatrix}
X_1\\X_2\\f\\y
\end{pmatrix} = \begin{pmatrix}
\begin{pmatrix}X_1\\\alpha_2\end{pmatrix}\\e\\u
\end{pmatrix}^\top\begin{pmatrix}
\begin{pmatrix}\alpha_1\\X_2\end{pmatrix}\\f\\y
\end{pmatrix} = 0
\end{align*}
and therefore $\mathcal{D}$ is recognised as a constant Dirac structure. Define further the subspace
\begin{align*}
\mathcal{R} := \set{(f,e)\,\big\vert\,e = Rf, f\in\mathbb{R}^{n_1+n_2}},
\end{align*}
which is non-negative by non-negativity of $R$ and Lagrangian by symmetry of $R$. Then, the port-Hamiltonian system
\begin{align*}
\big(\tfrac{\mathrm{d}}{\mathrm{d}t}z,f,-y,\nabla H(z),e,u\big)\in\mathcal{D},\qquad (f,e)\in\mathcal{R}
\end{align*}
is equivalently described by the differential-algebraic equation
\begin{align*}
\begin{pmatrix}
\partial_{z_1}H(z_1,z_2)\\
\tfrac{\mathrm{d}}{\mathrm{d}t}z_2
\end{pmatrix} & = J\begin{pmatrix}
\tfrac{\mathrm{d}}{\mathrm{d}t}z_1\\
\partial_{z_2}H(z_1,z_2)
\end{pmatrix} + e + Bu\\
f & = -\begin{pmatrix}
\tfrac{\mathrm{d}}{\mathrm{d}t}z_1\\
\partial_{z_2}H(z_1,z_2)
\end{pmatrix}\\
y & = B^\top \begin{pmatrix}
\tfrac{\mathrm{d}}{\mathrm{d}t}z_1\\
\partial_{z_2}H(z_1,z_2)
\end{pmatrix}\\
e & = Rf
\end{align*}
from which the latent variables $e$ and $f$ are removed to arrive at~\eqref{eb_1}.
\end{proof}

\begin{remark}\label{rem:non_uniqueness}
Of course, the presented representation of~\eqref{eb_1} as port-Hamil\-tonian system is not unique with the usual three sources of ambiguity:
\begin{itemize}
\item The Dirac structure is unknown outside of $\mathrm{graph}\,\nabla H$. This ambiguity is removed when we do not look at the single equation~\eqref{eb_1}, but instead leave $H$ itself free.
\item From the coordinate representation alone, it is unknown whether $u$ (and, by virtue of duality, $y$) is purely an element of $(\mathbb{R}^m)^*$ or an element of $\mathbb{R}^m$ or a mixed element; likewise for $e_r$ and $f_r$. This ambiguity is removed, when we consider uniqueness up to a splitting-preserving unitary automorphism $(\varphi\oplus\varphi^*)\oplus \psi_1\oplus\psi_2$, where $\varphi:\mathbb{R}^m\simto\mathbb{R}^m$, $\psi_1:\mathbb{R}^{n_1+n_2}\oplus\big(\mathbb{R}^{n_1+n_2}\big)^*\simto \mathbb{R}^{n_1+n_2}\oplus\big(\mathbb{R}^{n_1+n_2}\big)^*$ and $\psi_2:\mathbb{R}^m\oplus\big(\mathbb{R}^m\big)^*\simto \mathbb{R}^m\oplus\big(\mathbb{R}^m\big)^*$.
\item The third ambiguity is the dimension of the latent variables $f_r$ and $e_r$: If $A\in\mathbb{R}^{n_r\times (n_1+n_2)}$ has full column rank, then using the skew-symmetric matrix
\begin{align*}
\begin{bmatrix}
J & A^\top & B\\
-A & 0 & 0\\
-B^\top & 0 & 0
\end{bmatrix}
\end{align*}
in the Dirac structure and the resistive structure governed by $e = A(A^\top A)^{-1}R(A^\top A)^{-1}A^\top f$ gives likewise~\eqref{eb_1}. This ambiguity can be removed when we treat, similarly to the Hamiltonian function, the matrix $R$ as a parameter of the system so that we are interested in the Dirac structure up to unitary splitting-preserving automorphism with minimal dimension in the latent variables, which gives the port-Hamiltonian formulation of~\eqref{eb_1} for all $H$ and $R$.
\end{itemize}
\end{remark}

\section{Two distinct port-Hamiltonian formulations}

In this section, we present two distinct port-Hamiltonian formulations of the systems~\eqref{eq:desired_systems}. The distinction between the formulations lies in the r\^{o}le of the variable $z_3$. Either, $z_3$ may be interpreted as a free covector in a Lagrangian submanifold, when the zero in the equation corresponds to a constant auxiliary state, or the pair of $z_3$ and $0$ may be viewed as a trivial resistive port. 

\begin{proposition}\label{prop:implicit_representation}
The dynamical system~\eqref{eq:desired_systems}
is equivalent to a port-Hamil\-tonian with implicitly defined energy so that $z_3$ is a free co-vector.
\end{proposition}
\begin{proof}
Choose any nonempty open $V\subseteq\mathbb{R}^n$ and $a_0\in V$. On the extended space $\overline{U} := U\times V$, define
\begin{align*}
\mathcal{L} := \set{(z,a_0,\nabla H(z),z_3)\,\big\vert\,z\in U, z_3\in V}.
\end{align*}
Similarly to the graph of the derivative of a function it is verified that $\mathcal{L}$ is Lagrangian. Clearly, $\mathcal{L}$ is a submanifold of $\overline{U}\times\mathbb{R}^{n_1+n_2+n_3}$ with dimension $\dim\mathcal{L} = n_1+n_2+n_3$. Therefore, it remains to verify that the fibres of the tangent bundle $\mathcal{T}\mathcal{L}$ are isotropic. By chain rule, every tangent vector to $\mathcal{L}$ has the form
\begin{align*}
(X,0,\nabla^2 H(z)X,Y)\in\mathcal{T}_{(z,a_0,\nabla H(z),z_3)}\mathcal{L}.
\end{align*}
Therefore, the lemma of H.A.~Schwarz implies
\begin{align*}
& \begin{pmatrix}
X\\
0\\
\nabla^2 H(z)X\\
Y
\end{pmatrix}^\top\begin{pmatrix}
0 & 0 & -I_{n_1+n_2} & 0\\
0 & 0 & 0 & -I_{n_3}\\
I_{n_1+n_2} & 0 & 0 & 0\\
0 & I_{n_3} & 0 & 0\\
\end{pmatrix}\begin{pmatrix}
X'\\
0\\
\nabla^2 H(z)X'\\
Y'
\end{pmatrix}\\
&\hspace{1cm}  = X^\top\nabla^2 H(z)X' + 0^\top Y' - (X')^\top \nabla^2 H(z)X - Y^\top 0 = 0
\end{align*}
and hence $\mathcal{L}$ is indeed a Lagrangian submanifold. Alternatively, $\mathcal{L} \simeq \mathrm{graph}\,\nabla H\times\mathcal{T}^*_{a_0}V$ is (up to reordering of coordinates) the product of Lagrangian submanifolds and therefore a Lagrangian submanifold. With the Dirac structure $\mathcal{D}$ characterised by
\begin{equation*}
\begin{aligned}
\begin{pmatrix}
\begin{pmatrix}\alpha_1\\X_2\\X_3\end{pmatrix}\\f\\y
\end{pmatrix} = \begin{bmatrix}
J & I & B\\
-I & 0 & 0\\
-B^\top & 0 & 0
\end{bmatrix}\begin{pmatrix}
\begin{pmatrix}X_1\\\alpha_2\\\alpha_3\end{pmatrix}\\e\\u
\end{pmatrix}
\end{aligned}
\end{equation*}
for all $(X_1,X_2,X_3,f,y,\alpha_1,\alpha_2,\alpha_3,e,u)\in\mathcal{D}$
and the Lagrangian non-negative $\mathcal{R} := \set{(f,e)\,\big\vert\,e = Rf, f\in\mathbb{R}^{n_1+n_2+n_3}}$, the port-Hamiltonian system
\begin{equation}\label{eq:pH_1}
\begin{aligned}
\big(\tfrac{\mathrm{d}}{\mathrm{d}t}(z_1,z_2,a),f_r,-y,(e_1,e_2,z_3),e_r,u\big) & \in\mathcal{D},\\
(f_r,e_r) & \in\mathcal{R},\\
\big((z_1,z_2,a),(e_1,e_2,z_3)\big)& \in\mathcal{L}
\end{aligned}
\end{equation}
is, after elimination of the latent variables $e_r$ and $f_r$, which are uniquely determined by $\big(\tfrac{\mathrm{d}}{\mathrm{d}t}z_1,\partial_{z_2}H,z_3\big)$, equivalent to
\begin{align*}
\begin{pmatrix}
\partial_{z_1}H\\
\tfrac{\mathrm{d}}{\mathrm{d}t}z_2\\0
\end{pmatrix} & = (J-R)\begin{pmatrix}
\tfrac{\mathrm{d}}{\mathrm{d}t}z_1\\\partial_{z_2}H\\z_3
\end{pmatrix} + Bu\\
y & = B^\top \begin{pmatrix}
\tfrac{\mathrm{d}}{\mathrm{d}t}z_1\\\partial_{z_2}H(z_1,z_2)\\z_3
\end{pmatrix}\\
a & = a_0.
\end{align*}
The trivial equation $a = a_0$ may be removed, so that the system~\eqref{eq:pH_1} is recognised as equivalent to~\eqref{eq:desired_systems}.
\end{proof}

The implicit representation that we have chosen above treats $z_3$ as a free co-vector in the fibre $\mathcal{T}_{a_0}^*V$. The coordinates allow also for $z_3$ to be a free vector as follows.

\begin{proposition}\label{prop:semi-explicit_representation}
There exists a port-Hamiltonian system $\Sigma_{\mathrm{pH}}$ with explicitly defined energy on the extended state space $\overline{U} = U\times\mathbb{R}^{n_3}$ so that each solution $(z_1,z_2,z_3,u,y)$ of~\eqref{eq:desired_systems} is given as $(z_1,z_2,\tfrac{\mathrm{d}}{\mathrm{d}t}\zeta,u,y)$ for a solution $(z_1,z_2,\zeta,u,y)$ of $\Sigma_{\mathrm{pH}}$. 
\end{proposition}
\begin{proof}
Put $\widetilde{U} := \set{(z_1,\zeta,z_2)\,\big\vert\,(z_1,z_2)\in U, \zeta\in\mathbb{R}^{n_3}}$ and put $\widetilde{H}$ as the pullback of $H$ to $\widetilde{U}$ under the canonical projection. In Proposition~\ref{prop:es_pH}, we demonstrated that
\begin{align*}
\begin{pmatrix}
\partial_{(z_1,\zeta)}\widetilde{H}\\\tfrac{\mathrm{d}}{\mathrm{d}t}z_2
\end{pmatrix} & = (\widetilde{J}-\widetilde{R})\begin{pmatrix}
\tfrac{\mathrm{d}}{\mathrm{d}t}\begin{pmatrix}
z_1\\\zeta
\end{pmatrix}\\
\partial_{z_2}\widetilde{H}
\end{pmatrix} + \widetilde{B}u,\\
y & = \widetilde{B}^\top \begin{pmatrix}
\tfrac{\mathrm{d}}{\mathrm{d}t}\begin{pmatrix}
z_1\\\zeta
\end{pmatrix}\\
\partial_{z_2}\widetilde{H}
\end{pmatrix},
\end{align*}
where
\begin{align*}
\widetilde{J} = TJT^\top,\ \widetilde{R} = TRT^\top,\ \widetilde{B} = TB,\ T = \begin{bmatrix}
I_{n_1} & 0 & 0\\
0 & 0 & I_{n_3}\\
0 & I_{n_2} & 0
\end{bmatrix}.
\end{align*}
Applying the unitary transformation $(T^\top \oplus T)\oplus (T^\top \oplus T)\oplus I$ to the Dirac structure, and $T^\top\oplus T$ to the resistive structure of that system as constructed in the proof of Proposition~\ref{prop:es_pH} gives that the system
\begin{equation*}
\begin{aligned}
\begin{pmatrix}
\partial_{z_1}H\\
\tfrac{\mathrm{d}}{\mathrm{d}t}z_2\\0
\end{pmatrix} & = (J-R)\begin{pmatrix}
\tfrac{\mathrm{d}}{\mathrm{d}t}z_1\\\partial_{z_2}H\\\tfrac{\mathrm{d}}{\mathrm{d}t}\zeta
\end{pmatrix} + Bu\\
y & = B^\top \begin{pmatrix}
\tfrac{\mathrm{d}}{\mathrm{d}t}z_1\\\partial_{z_2}H\\z_3
\end{pmatrix}
\end{aligned}
\end{equation*}
is equivalent to a port-Hamiltonian system, whose solution satisfy the proposed projection property.
\end{proof}

Notably, Proposition~\ref{prop:semi-explicit_representation} does not give a port-Hamiltonian representation of~\eqref{eq:desired_systems}, since the equivalence of the systems is only up to the unknown and undetermined initial value of $\zeta$. The alternative representation of~\eqref{eq:desired_systems} as port-Hamiltonian system arises when the variable $z_3$ is considered as a part of the resistive port-variables, which does not contribute to the dissipation, and which can not be eliminated from the system.

\begin{proposition}\label{prop:explicit_representation}
The dynamical system~\eqref{eq:desired_systems} is equivalent to a port-Hamil\-tonian with explicitly defined energy $H$. 
\end{proposition}
\begin{proof}
Define $\mathcal{R}\leq\mathbb{R}^{n_3+(n_1+n_2)+n_3}\oplus\big(\mathbb{R}^{n_3+(n_1+n_2)+n_3}\big)^*$ by
\begin{align*}
\mathcal{R} := \set{(0,f_2,f_3,e_1,e_2,e_3)\,\left\vert\,\begin{pmatrix}
e_2\\e_3
\end{pmatrix} = R\begin{pmatrix}
f_2\\f_3
\end{pmatrix}\right.}.
\end{align*}
Clearly, $\dim\mathcal{R} = n_1+n_2+2n_3$, and since
\begin{align*}
\begin{pmatrix}
e_1\\e_2\\e_3
\end{pmatrix}^\top\begin{pmatrix}
0\\f_2'\\f_3'
\end{pmatrix}-\begin{pmatrix}
e_1'\\e_2'\\e_3'
\end{pmatrix}^\top\begin{pmatrix}
0\\f_2\\f_3
\end{pmatrix} = \begin{pmatrix}
f_2\\f_3
\end{pmatrix}^\top R\begin{pmatrix}
f_2'\\f_3'
\end{pmatrix} - \begin{pmatrix}
f_2'\\f_3'
\end{pmatrix}^\top R\begin{pmatrix}
f_2\\f_3
\end{pmatrix} = 0,
\end{align*}
$\mathcal{R}$ is recognised as a Lagrangian subspace. Consider further the Dirac structure $\mathcal{D}$ characterised by
\begin{equation}\label{eq:explicite_Dirac}
\begin{aligned}
\begin{pmatrix}
\begin{pmatrix}
\begin{pmatrix}\alpha_1\\X_2\end{pmatrix}\\f_1
\end{pmatrix}\\\begin{pmatrix}
f_2\\f_3
\end{pmatrix}\\y
\end{pmatrix} = \begin{bmatrix}
J & I & B\\
-I & 0 & 0\\
-B^\top & 0 & 0
\end{bmatrix}\begin{pmatrix}
\begin{pmatrix}
\begin{pmatrix}X_1\\\alpha_2\end{pmatrix}\\e_1
\end{pmatrix}\\\begin{pmatrix}
e_2\\e_3
\end{pmatrix}\\u
\end{pmatrix}.
\end{aligned}
\end{equation}
The port-Hamiltonian system
\begin{align*}
\big(\tfrac{\mathrm{d}}{\mathrm{d}t}z_1,\tfrac{\mathrm{d}}{\mathrm{d}t}z_2,f_1,f_2,f_3,-y,\partial_{z_1}H,\partial_{z_2}H,e_1,e_2,e_3,u\big) & \in\mathcal{D}\\
\big(f_1,f_2,f_3,e_1,e_2,e_3\big) & \in\mathcal{R}
\end{align*}
is, due to $f_1 = 0$, $-f_2 = \begin{pmatrix}X_1\\\alpha_2\end{pmatrix}$ and $f_3 = e_1$, then equivalent to
\begin{align*}
\begin{pmatrix}
\partial_{z_1}H\\
\tfrac{\mathrm{d}}{\mathrm{d}t}z_2\\0
\end{pmatrix} & = (J-R)\begin{pmatrix}
\tfrac{\mathrm{d}}{\mathrm{d}t}z_1\\\partial_{z_2}H\\e_1
\end{pmatrix} + Bu\\
y & = B^\top \begin{pmatrix}
\tfrac{\mathrm{d}}{\mathrm{d}t}z_1\\\partial_{z_2}H\\e_1
\end{pmatrix}
\end{align*}
and renaming $z_3 := e_1$ recovers~\eqref{eq:desired_systems}.
\end{proof}

The author would argue that this, i.e. the recognition of $z_3$ as a port-variable, is the true port-Hamiltonian formulation of~\eqref{eq:desired_systems}. In particular, this formulation points out the possibility of identifying some parts of the resistive port as constrained external port-variables, when we view $z_3$ not as part of the ``state'', but instead as part of the control.

In the representation presented in the proof of Proposition~\ref{prop:explicit_representation}, $z_3$ is, again, a co-vector. As discussed in Remark~\ref{rem:non_uniqueness}, this is a non-canonical choice, and $z_3$ can also be chosen as a vector or a combination of vectors and co-vectors. This distinction is hidden in~\eqref{eq:desired_systems} behind the identification of vectors and co-vectors with elements in the coordinate space.

\section{Structure-preserving interconnection}

Altmann and Schulze also consider a structure-preserving interconnection, which is composed from a power-preserving and a dissipative interconnection. From the port-Hamiltonian systems point of view, this is the three-fold interconnection of two systems~\eqref{eq:desired_systems} and a port-Hamiltonian system, which has only resistive and external ports. Taking the direct sum of the two systems~\eqref{eq:desired_systems} results, after reordering the variables, in a single system~\eqref{eq:desired_systems} and the interconnection becomes a two-fold interconnection. Surely, this interconnection preserves the port-Hamiltonian structure, cf.~\cite{Interconnection07}, but the specific ``graphical'' structure of the Dirac structure is not automatically preserved. To give a sufficient condition, we characterise the Dirac structures given in the proof of Proposition~\ref{prop:explicit_representation} abstractly.

\begin{lemma}\label{lem:abstract_Dirac}
Let $n = n_1+n_2$, and consider a Dirac structure
\begin{align*}
\mathcal{D}\leq\mathbb{R}^{n_1+n_2+n_3+n+m}\oplus\mathbb{R}^{n_1+n_2+n_3+n+m}
\end{align*}
with the properties
\begin{enumerate}[(i)]
\item $p_{1,7,8,9,10}\mathcal{D} = \mathbb{R}^{n_1}\oplus\mathbb{R}^{n_2}\oplus\mathbb{R}^{n_3}\oplus\mathbb{R}^n\oplus\mathbb{R}^m$
\item $p_{4,5}\big(\mathcal{D}\cap p_{1,6,7}^{-1}(\set{(0,0,0)})\big) = \set{(0,0)}$
\end{enumerate}
Given a symmetric and non-negative $R\in\mathbb{R}^{n\times n}$, the port-Hamiltonian system
\begin{align*}
\big(\tfrac{\mathrm{d}}{\mathrm{d}t}z_1,\tfrac{\mathrm{d}}{\mathrm{d}t}z_2,f_1,f_2,-y,\partial_{z_1} H,\partial_{z_2}H,e_1,e_2,u\big) & \in\mathcal{D},\\
f_1 & = 0,\\
e_2 & = Rf_2
\end{align*}
is equivalent to a system~\eqref{eq:desired_systems}.
\end{lemma}
\begin{proof}
By assumption (i) and since $\mathcal{D}$ is a Dirac structure, there is $J_{\mathrm{tot}}\in\mathbb{R}^{(n+n_3+n+m)\times(n+n_3+n+m)}$ with $J_{\mathrm{tot}}+J_{\mathrm{tot}}^\top = 0$ and
\begin{align*}
\mathcal{D} = \set{(X_1,X_2,f_1,f_2,y,\alpha_1,\alpha_2,e_1,e_2,u)\,\left\vert\,\begin{pmatrix}
\alpha_1\\X_2\\f_1\\f_2\\y
\end{pmatrix} = J_{\mathrm{tot}}\begin{pmatrix}
X_1\\\alpha_2\\e_1\\e_2\\u
\end{pmatrix}\right.}.
\end{align*}
Write $J_{\mathrm{tot}}$ as a block-matrix
\begin{align*}
J_{\mathrm{tot}} = \begin{bmatrix}
J & A & B\\
-A^\top & J_1 & C\\
-B^\top & -C^\top & J_2
\end{bmatrix}
\end{align*}
with $J\in\mathbb{R}^{(n+n_3)\times (n+n_3)}$, $J_1\in\mathbb{R}^{n\times n}$ and $J_2\in\mathbb{R}^{m\times m}$. Then, assumption (ii) implies $J_1 = 0$, $C = 0$ and $J_2 = 0$. Plugging in the port-Hamiltonian system gives therefore the equation
\begin{align*}
\begin{pmatrix}
\partial_{z_1}H\\
\tfrac{\mathrm{d}}{\mathrm{d}t}z_2\\
0
\end{pmatrix} & = (J-ARA^\top)\begin{pmatrix}
\tfrac{\mathrm{d}}{\mathrm{d}t}z_1\\
\partial_{z_2}H\\
e_1
\end{pmatrix} + Bu\\
y & = B^\top \begin{pmatrix}
\tfrac{\mathrm{d}}{\mathrm{d}t}z_1\\
\partial_{z_2}H\\
e_1
\end{pmatrix}
\end{align*}
which has indeed the structure~\eqref{eq:desired_systems}.
\end{proof} 

We may now give a sufficient condition for the interconnection to be structure preserving.

\begin{proposition}
Let $\mathcal{D}_{\mathrm{int}}\leq\mathbb{R}^{m+p_r+p}\oplus\mathbb{R}^{m+p_r+p}$ be a Dirac structure, and consider the interconnected system
\begin{align*}
\begin{pmatrix}
\partial_{z_1}H(z_1,z_2)\\
\tfrac{\mathrm{d}}{\mathrm{d}t}z_2\\
0
\end{pmatrix} & = (J-R)\begin{pmatrix}
\tfrac{\mathrm{d}}{\mathrm{d}t}z_1\\
\partial_{z_2}H(z_1,z_2)\\
z_3
\end{pmatrix} + Bu\\
y & = B^\top\begin{pmatrix}
\tfrac{\mathrm{d}}{\mathrm{d}t}z_1\\
\partial_{z_2}H(z_1,z_2)\\
z_3
\end{pmatrix}\\
(y,f,-\widetilde{y},u,e,\widetilde{u}) & \in\mathcal{D}_{\mathrm{int}}\\
e & = R_{\mathrm{int}}f
\end{align*}
If, for all $(X,\alpha,e_1)\in\mathbb{R}^{n_1+n_2+n_3}$ and $(e,\widetilde{u})\in\mathbb{R}^{p_r+p}$:
\begin{enumerate}[(a)]
\item there exists $u\in\mathbb{R}^m$ with
\begin{align*}
\left(B^\top\begin{pmatrix}
X\\\alpha\\e_1
\end{pmatrix},u,e,\widetilde{u}\right)\in p_{1,4,5,6}\mathcal{D}_{\mathrm{int}},
\end{align*}
\item whenever
\begin{align*}
\left(B^\top\begin{pmatrix}
X\\\alpha\\e_1
\end{pmatrix},f,\widetilde{y},u,e,\widetilde{u}\right),\left(B^\top\begin{pmatrix}
X\\\alpha\\e_1
\end{pmatrix},f',\widetilde{y}',u,e,\widetilde{u}\right)\in\mathcal{D}_{\mathrm{int}},
\end{align*}
$f = f'$ and $\widetilde{y} = \widetilde{y}'$,
\item there exists $(f,\widetilde{y},u)\in\mathbb{R}^{p_r+p+m}$ with
\begin{align*}
\begin{pmatrix}
0,f,\widetilde{y},u,e,\widetilde{u}
\end{pmatrix}\in \mathcal{D}_{\mathrm{int}},
\end{align*}
\end{enumerate}
then the interconnected system is equivalent to a system~\eqref{eq:desired_systems}.
\end{proposition}
\begin{proof}
Let $\mathcal{D}$ denote the Dirac structure characterised by~\eqref{eq:explicite_Dirac}. The Dirac structure $\overline{\mathcal{D}}$ describing the interconnected system is characterised by the relation $(X_1,X_2,f_1,f_2,f_3,\widetilde{y},\alpha_1,\alpha_2,e_1,e_2,e_3,\widetilde{u})\in \overline{\mathcal{D}}$ if, and only if,
\begin{align*}
(X_1,X_2,f_1,f_2,-y,\alpha_1,\alpha_2,e_1,e_2,u)& \in\mathcal{D}\\
(y,f_3,\widetilde{y},u,e_3,\widetilde{u} & \in\mathcal{D}_{\mathrm{int}}
\end{align*}
for some $y,u\in\mathbb{R}^m$. Then, (a) implies (i) of Lemma~\ref{lem:abstract_Dirac}, and (b) and (c) imply (ii) of Lemma~\ref{lem:abstract_Dirac}.
\end{proof}

\begin{example}
The change of the dimension of the external port is a structure-preserving interconnection: Given a system~\eqref{eq:desired_systems}, then the interconnection constraint
\begin{align*}
u = D\widetilde{u},\qquad \widetilde{y} = D^\top y
\end{align*}
for some $D\in\mathbb{R}^{m\times m'}$ corresponding to the interconnection Dirac structure
\begin{align*}
\mathcal{D}_{\mathrm{int}} = \set{(u,\widetilde{u},y,\widetilde{y})\,\left\vert\,\begin{pmatrix}
u\\\widetilde{y}
\end{pmatrix} =  \begin{bmatrix}
0 & D\\
-D^\top & 0
\end{bmatrix}\begin{pmatrix}
y\\\widetilde{u}
\end{pmatrix}\right.}
\end{align*}
gives the system
\begin{align*}
\begin{pmatrix}
\partial_{z_1}H(z_1,z_2)\\
\tfrac{\mathrm{d}}{\mathrm{d}t}z_2\\0
\end{pmatrix} & = (J-R)\begin{pmatrix}
\tfrac{\mathrm{d}}{\mathrm{d}t}z_1\\\partial_{z_2}H(z_1,z_2)\\z_3
\end{pmatrix} + BD\widetilde{u}\\
\widetilde{y} & = D^\top B^\top \begin{pmatrix}
\tfrac{\mathrm{d}}{\mathrm{d}t}z_1\\\partial_{z_2}H(z_1,z_2)\\z_3
\end{pmatrix}
\end{align*}
\end{example}

\section{Adding the feedthrough term}

Notably, we have omitted feedthrough, so that the system class does not quite encompasses the es-ph systems. It is, however, straightforward to consider systems with feedthrough as
\begin{align}\label{eq:feedthrough}
\begin{pmatrix}
\partial_{z_1}H\\
\tfrac{\mathrm{d}}{\mathrm{d}t}z_2\\
0\\
-y
\end{pmatrix} = (J-R) \begin{pmatrix}
\tfrac{\mathrm{d}}{\mathrm{d}t}z_1\\
\partial_{z_2}H\\
z_3\\
u
\end{pmatrix}
\end{align}
where $J,R\in\mathbb{R}^{(n_1+n_2+n_3+m)\times(n_1+n_2+n_3+m)}$ with $J + J^\top = R - R^\top = 0$ and $R\geq 0$. In case $n_3 = 0$, the systems~\eqref{eq:feedthrough} coincide with the energy-stable port-Hamiltonian system of~\cite{BuchGlasZwar26}, so that the system class~\eqref{eq:feedthrough} are a true generalisation of the latter. It is readily demonstrated that the systems~\eqref{eq:feedthrough} admit a similar double representation as port-Hamiltonian systems as the feedthrough-free systems~\eqref{eq:desired_systems}.

\begin{proposition}
The system~\eqref{eq:feedthrough} may be represented as 
\begin{enumerate}[(i)]
\item port-Hamiltonian system with explicitly defined energy and
\item port-Hamiltonian system with implicitly defined energy.
\end{enumerate}
\end{proposition}
\begin{proof}
Consider the Dirac structure $\mathcal{D}$ characterised by
\begin{align}\label{eq:feedthrough_Dirac}
\begin{pmatrix}
\alpha_1\\
X_2\\
f_1\\
y\\
f_2
\end{pmatrix} = \begin{bmatrix}
J & I\\
-I & 0
\end{bmatrix}\begin{pmatrix}
X_1\\\alpha_2\\e_1\\u\\e_2
\end{pmatrix}
\end{align}
and the non-negative Lagrangian subspace
\begin{align*}
\mathcal{R} := \set{(0,f_2,e_1,e_2)\,\big\vert\,e_1\in\mathbb{R}^{n_3}, e_2 = Rf_2}.
\end{align*}
In the port-Hamiltonian system
\begin{align*}
\begin{pmatrix}
\tfrac{\mathrm{d}}{\mathrm{d}t}z_1,\tfrac{\mathrm{d}}{\mathrm{d}t}z_2,f_1,f_2,-y,\partial_{z_1}H,\partial_{z_2}H,e_1,e_2,u\end{pmatrix} & \in\mathcal{D}\\
\begin{pmatrix}
f_1,f_2,e_1,e_2
\end{pmatrix} & \in\mathcal{R},
\end{align*}
the latent variables $f_2$ and $e_2$ may be eliminated and $f_1 = 0$ by definition of $\mathcal{R}$. This leads to
\begin{align*}
\begin{pmatrix}
\partial_{z_1}H\\
\tfrac{\mathrm{d}}{\mathrm{d}t}z_2\\
0\\
-y
\end{pmatrix} = (J-R) \begin{pmatrix}
\tfrac{\mathrm{d}}{\mathrm{d}t}z_1\\
\partial_{z_2}H\\
e_1\\
u
\end{pmatrix}
\end{align*}
and renaming $z_3 := e_1$ gives~\eqref{eq:feedthrough}. Choose further an open $V\subseteq\mathbb{R}^{n_3}$ and $a_0\in V$, and consider the Lagrangian submanifold
\begin{align*}
\mathcal{L} := \set{(z,a_0,\nabla_z H(z),0)\,\big\vert\,z\in U},
\end{align*}
the non-negative Lagrangian subspace $\mathcal{R}' := \set{(f,e)\,\big\vert\,e = Rf}$ and the Dirac structure $\mathcal{D}'$ characterised by~\eqref{eq:feedthrough_Dirac} with $e_1$ and $f_1$ interchanged. Then, the port-Hamiltonian system
\begin{align*}
\begin{pmatrix}
\tfrac{\mathrm{d}}{\mathrm{d}t}z_1,\tfrac{\mathrm{d}}{\mathrm{d}t}z_2,a,f_r,-y,e^1,e^2,e^3,e_r,u\end{pmatrix} & \in\mathcal{D}'\\
\begin{pmatrix}
f_r,e_r
\end{pmatrix} & \in\mathcal{R}',\\
\begin{pmatrix}
(z_1,z_2,a,e^1,e^2,e^3
\end{pmatrix} & \in\mathcal{L}
\end{align*}
is, by elimination of the latent variables $f_r$ and $e_r$, equivalent to
\begin{align*}
\begin{pmatrix}
\partial_{z_1}H\\
\tfrac{\mathrm{d}}{\mathrm{d}t}z_2\\
0\\
-y
\end{pmatrix} & = (J-R) \begin{pmatrix}
\tfrac{\mathrm{d}}{\mathrm{d}t}z_1\\
\partial_{z_2}H\\
e^3\\
u
\end{pmatrix}, \qquad a = a_0.
\end{align*}
Elimination of the trivial equation $a = a_0$ and renaming $z_3 := e^3$ gives~\eqref{eq:feedthrough}.
\end{proof}

\section{Conclusion}

We have proposed port-Hamiltonian representations for the systems~\eqref{eq:desired_systems} recognising the modelling approach of~\cite{AltmSchu25} as a subclass of port-Hamiltonian systems theory. The representations are not unique. We have recalled the three sources of ambiguity in the geometric representation of the coordinate representation for classical and energy-stable port-Hamiltonian systems. For the systems~\eqref{eq:desired_systems}, a fourth ambiguity arises, which can not be removed by simple argumentation. Analogous to~\cite{BuchGlasZwar26}, we extended the system class within port-Hamiltonian systems by adding feedthrough.

\bibliographystyle{plain}

\end{document}